\documentclass[12pt]{amsart}
\usepackage{fancyhdr}
\usepackage[latin1, utf8]{inputenc}
\usepackage{amsbsy}
\usepackage{amscd} 
\usepackage{amsfonts}
\usepackage{amsgen} 
\usepackage{amsmath}
\usepackage{amsopn} 
\usepackage{amssymb}
\usepackage{amstext}
\usepackage{amsthm} 
\usepackage{amsxtra}
\usepackage{hyperref}
\usepackage[titlenumbered, linesnumbered, ruled, vlined]{algorithm2e}
\usepackage[english]{babel}

\usepackage{tikz}
\usetikzlibrary{positioning}
\usepackage{longtable}

\usepackage{float}

\theoremstyle{plain} 
\newtheorem{thm}{Theorem}[section] 

\newtheorem*{thm*}{Main Theorem}

\newtheorem{lem}[thm]{Lemma}

\theoremstyle{definition}
\newtheorem{defn}[thm]{Definition}
\newtheorem{rem}[thm]{Remark}
\newtheorem{ex}[thm]{Example}

\numberwithin{equation}{section}

\renewcommand{\theta}{\vartheta}
\renewcommand{\phi}{\varphi}
\renewcommand{\epsilon}{\varepsilon}
\renewcommand{\subset}{\subseteq}

\newcommand{\N}{\mathbb N}
\newcommand{\Z}{\mathbb Z}

\newcommand{\C}{\mathbb C}

\newcommand{\norm}[1]{\lVert {#1}\rVert}

\newcommand{\Aut}{G_{\textnormal{aut}}}
\newcommand{\AAut}{A^+_{\textnormal{aut}}}
\newcommand{\QBan}{G_{\textnormal{aut}}^+}

\begin{document} 
\title{Existence of quantum symmetries for graphs on up to seven vertices: a computer based approach}
\author{Christian Eder, Viktor Levandovskyy, Julien Schanz, Simon Schmidt, Andreas Steenpass, Moritz Weber}
\address{Ch. E., 
Universit\"at Leipzig,
Fakult\"at f\"ur Mathematik und Informatik,
Augustusplatz 10,
04109 Leipzig, Germany}
\email{ederc@mathematik.uni-kl.de}

\address{V. L., 
RWTH Aachen University, 
Lehrstuhl D f\"ur Mathematik,
D-52062 Aachen, Germany} 
\email{viktor.levandovskyy@math.rwth-aachen.de}

\address{J. S., S. S., M. W., Saarland University, Fachbereich Mathematik, Postfach 151150,
66041 Saarbr\"ucken, Germany}
\email{s8juscha@stud.uni-saarland.de}
\email{simon.schmidt@math.uni-sb.de}
\email{weber@math.uni-sb.de}

\address{A. St., TU Kaiserslautern, Fachbereich Mathematik, Gottlieb-Daimler-Stra\ss e, Geb\"aude 48, Raum 425,
67663 Kaiserslautern, Germany}
\email{steenpass@mathematik.uni-kl.de}

\thanks{This work has been supported by the SFB-TRR 195. In particular, we would like to acknowledge the workshop \emph{Introduction to computer algebra systems}, Saarbr\"ucken 2017, funded by the SFB-TRR 195, where this work was initiated. Simon Schmidt and Moritz Weber were supported by the DFG grant \emph{Quantenautomorphismen von Graphen}.}
\date{\today}

\begin{abstract}
The symmetries of a finite graph are described by its automorphism group; in the setting of Woronowicz's quantum groups, a notion of a quantum automorphism group has been defined by Banica capturing the quantum symmetries of the graph. In general, there are more quantum symmetries than symmetries and it is a non-trivial task to determine when this is the case for a given graph: The question is whether or not the algebra associated to the quantum automorphism group is commutative. We use Gr\"obner base computations in order to tackle this problem; the implementation uses \textsc{Gap} and the \textsc{Singular} package \textsc{Letterplace}. We determine the existence of quantum symmetries for all connected, undirected graphs without multiple edges and without self-edges, for up to seven vertices. As an outcome, we infer within our regime that a classical automorphism group of order one or two is an obstruction for the existence of quantum symmetries.
\end{abstract}

\maketitle

\section{Introduction}
Given a finite graph $\Gamma=(V,E)$ on $n$ vertices having no multiple edges (i.e. we have $E\subset V\times V$), we denote its adjacency matrix by $\epsilon\in M_n(\{0,1\})$. The 
automorphism group is a subgroup of the symmetric group $S_n$ given by
\[\Aut(\Gamma) = \left\{ \sigma \in S_n\; |\; \sigma \varepsilon = \varepsilon\sigma \right\}\subset S_n.\] 

In the framework of  compact 
matrix quantum groups, which were introduced by Woronowicz in~\cite{woronowicz1987}, Wang \cite{wang} defined the quantum symmetric group $S_n^+$ in terms of its associated  universal unital $C^*$-algebra
\[C(S_n^+) := C^*\langle u_{ij}, 1\leq i,j\leq n\;|\; u_{ij} = u_{ij}^* = u_{ij}^2, \sum_{k=1}^{n} u_{ik} = \sum_{k=1}^n u_{ki} = 1\rangle.\]
The quantum automorphism group of $\Gamma$ has been defined by Banica \cite{banica} via
\[C(\QBan(\Gamma)) := C^*\langle u_{ij}, 1\leq i,j\leq n\;|\; u_{ij} = u_{ij}^*=u_{ij}^2, \sum_{k=1}^n u_{ik} = \sum_{k=1}^n u_{ki} = 1, u \varepsilon = \varepsilon u\rangle.\]
If we interpret $\Aut(\Gamma)$ as a compact matrix quantum group, we see, that 
\[\Aut(\Gamma) \subseteq \QBan(\Gamma)\]
 holds for all graphs $\Gamma$.
The question  is, whether this is a strict inclusion. If it is so, we say that the graph has quantum symmetries. This is the case if and only if the algebra $C(\QBan(\Gamma))$ is non-commutative.
For many graphs it is not known, whether they have quantum symmetries -- we even do not know whether an asymmetric graph (i.e. $\Aut(\Gamma)=\{e\}$) can have quantum symmetries. There is some ``asymptotic'' evidence that no asymmetric graph has quantum symmetries, see \cite{LMR} and our results below support this hypothesis for small graphs.

This article reports on a computer based approach to the question of existence of quantum symmetries for a given graph $\Gamma=(V,E)$.
 We implemented in \textsc{Singular:Letterplace} \cite{singular, letterplace} an algorithm using Gr\"obner bases for checking whether or not the complex unital algebra $\AAut(\Gamma)$ generated by elements $ u_{ij}, 1\leq i,j\leq n$ and the following relations is non-commutative: 
\[
         u_{ij} u_{ik} = \delta_{jk} u_{ij},  u_{ji} u_{ki} = \delta_{jk} u_{ji},  \sum_{k=1}^n u_{ik} = \sum_{k=1}^n u_{ki} = 1,
        u_{ik} u_{jl} = u_{jl} u_{ik} = 0,  \textnormal{if }\epsilon_{ij}\neq\epsilon_{kl}
\]
Note that the canonical map from $\AAut(\Gamma)$ onto  $C(\QBan(\Gamma))$ has a dense image, so commutativity of $\AAut(\Gamma)$ implies the absence of quantum symmetries. 
We also make use of a criterion by one of the authors~\cite{simonFolded} that yields  $\Aut(\Gamma) \neq \QBan(\Gamma)$ if $\Aut(\Gamma)$ contains a pair of disjoint automorphisms and 
of another one by Fulton~\cite{fulton} which states that $u_{ij} = 0$ if $\varepsilon^{(l)}_{ii} \neq \varepsilon^{(l)}_{jj}$ holds for some power $\epsilon^l$ of the adjacency matrix.   These criteria are checked using \textsc{Gap} \cite{gap}.

We produced the following data on the amount of connected undirected graphs (without multiple edges and loops) having quantum symmetry.

\begin{table}[H]
        \centering\label{table1}
        \begin{tabular}{|c||c c|c c|c c|c c|}
                \hline
                Order&\multicolumn{2}{c|}{4 vertices}&\multicolumn{2}{c|}{5 vertices} & \multicolumn{2}{c|}{6 vertices} & \multicolumn{2}{c|}{7 vertices}\\
                of $\Aut(\Gamma)$&total& qsym&total& qsym&total& qsym&total& qsym\\
                \hline
                720&&& & &1&1&&\\
                120&&&1&1&1&1&&\\
                72&&&&&1&1&&\\
                48&&&&&4&4&&\\
                36&&&&&1&1&&\\
                24&1&1&1&1&1&1&&\\
                16&&&&&3&3&&\\
                12&&&3&3&10&8&&\\
                10&&&1&0&1&0&&\\
                8&1&1&2&2&9&9&&\\
                6&1&0&1&0&7&0&&\\
                4&1&1&3&3&28&26&&\\
                2&2&0&9&0&37&0&317&0\\
                1&0&0&0&0&8&0&144&0\\
                total&6&3&21&10&112&55&853&?\\
                \hline
        \end{tabular}
        \caption{Number of connected, undirected graphs with quantum symmetries}
\end{table} 

 From the above data, we immediately infer the following result; here $\mathbb Z_2$ denotes the cyclic group on two generators.
\begin{thm*}[{Thm. \ref{mainThm}}]
        Let $\Gamma$ be an undirected graph on $n\leq 7$ vertices having no multiple edges and no loops (i.e. $(i,i)\notin E$ for all $i$). Then:
        \begin{displaymath}
                \Aut(\Gamma) = \mathbb{Z}_2 \qquad \Rightarrow \qquad \QBan\left( \Gamma \right) = \mathbb{Z}_2
        \end{displaymath}
        \begin{displaymath}
                \Aut(\Gamma) = \left\{  e\right\}\qquad  \Rightarrow \qquad \QBan\left( \Gamma \right) = \left\{ e\right\}
        \end{displaymath}
\end{thm*}

 In Section \ref{List} we list all connected, undirected graphs on up to six vertices (having no multiple edges and no loops) and their information on symmetry and quantum symmetry.

\section{Quantum symmetries of graphs}
We first sketch the mathematical background of this article.

\subsection{$C^*$-algebras and quantum spaces}

A $C^*$-algebra is a complex, associative algebra $A$ equipped with an involution $^*:A\to A$ (an antilinear map with $(xy)^*=y^*x^*$ and $(x^*)^*=x$) and a norm with respect to which it is complete as a topological space. Moreover, the norm is required to satisfy $\norm{xy}\leq\norm{x}\norm{y}$ and $\norm{x^*x}=\norm{x}^2$, the latter one being the most characteristic property of a $C^*$-algebra which distinguishes it from the more general $^*$-Banach algebras.

Examples of $C^*$-algebras are the (unital) algebra $C(X)$ of complex-valued continuous functions on a compact, topological Hausdorff space $X$; here the multiplication is the pointwise multiplication of functions, the involution is the pointwise complex conjugation and the norm is the supremum norm of functions. Another example is the (unital) algebra $M_n(\C)$ of complex-valued $n\times n$ matrices equipped with the matrix multiplication, the adjoint of matrices and the matrix norm. We observe that $C(X)$ is always a commutative $C^*$-algebra and the converse is  a fundamental theorem in the theory of $C^*$-algebras: Actually \emph{all} commutative (unital) $C^*$-algebras arise in exactly this form. Hence, we may identify commutative $C^*$-algebras with compact, topological spaces and in this sense, the theory of noncommutative $C^*$-algebras may be viewed as a kind of noncommutative topology or of ``quantum spaces''  --  a point of view which turns out to be very fruitful, see for instance \cite{blackadar, graciabondia}. It is therefore common in the theory of $C^*$-algebras to define a quantum object via its associated (possibly noncommutative) $C^*$-algebra -- a philosophy which might be slightly disturbing for an outsider, but which is extremely instructive for people actually working in the field. In any case, all statements on quantum spaces are made precise in terms of their underlying $C^*$-algebra, so there is never any issue of mathematical precision.

A very abstract but also very useful construction of $C^*$-algebras is the one of universal $C^*$-algebras, see for instance \cite{weberindian} for a short introduction. The main idea is to take a free algebra on some set of generators $x$ and their adjoints $x^*$, to divide out the ideal generated by some polynomial relations, and to endow it with the supremum of all $C^*$-seminorms (which in turn may be obtained from representations of the abstract $^*$-algebra on the algebra of bounded operators on Hilbert spaces).

\begin{ex}\label{ExProj}
 An example is the universal unital $C^*$-algebra generated by two (orthogonal) projections:
\[C^*\langle p,q\;|\; p=p^2=p^*, q=q^2=q^*\rangle\]
Note that this $C^*$-algebra is noncommutative, since we may easily find two matrices $p,q\in M_2(\C)$ satisfying the above relations, and in addition $pq\neq qp$.
\end{ex}

\subsection{Connection with $\mathbb C$-algebras}\label{CAlgebras}

As described in the previous section, $C^*$-algebras are complex algebras with an additional structure: there is an involution and also a topological hull coming from a norm. In this article, we will use the computer to produce data on certain algebraic approximations of $C^*$-algebras (in a relatively weak sense). More precisely, given a set of generators $\mathcal E$ and a set of algebraic relations $\mathcal R$ in the generators and their adjoints, we consider the set of relations $\mathcal R'$ arising from $\mathcal R$ by omitting all relations involving the involution. We then study the universal complex algebra generated by the generators $\mathcal E$ and the relations $\mathcal R'$. In case this complex algebra is commutative, we may infer commutativity of the corresponding universal $C^*$-algebra generated by $\mathcal E$ and $\mathcal R$. See Section \ref{TwoCriteria} for a precise statement adapted to our situation.

\subsection{Compact matrix quantum groups and quantum symmetries}
\label{cmqg}
Compact matrix quantum groups were defined by Woronowicz~\cite{woronowicz1987} in 1987 in order to provide an appropriate notion of (quantum) symmetry, for instance for the above sketched quantum spaces; see also \cite{NT, Tim} for more on this subject.
A compact matrix quantum group $G=(A,u)$ is given by a unital $C^*$-algebra $A$  and a matrix $u=(u_{ij}) \in M_n(A)$, $n\in \N$, such that 
\begin{enumerate}
        \item there is a $*$-homomorphism $\Delta: A \rightarrow A \otimes A$ with $\Delta(u_{ij}) = \sum_k u_{ik} \otimes u_{kj}$ for all $i, j$,
        \item $u$ and $\bar{u}=(u_{ij}^*)$ are invertible matrices,
        \item and the elements $u_{ij}$ $(1 \leq i, j \leq n)$ generate $A$ as a $C^*$-algebra.
\end{enumerate}
Any compact group $G\subset GL_n(\mathbb C)$ gives rise to a compact matrix quantum group by identifying $G$ with $(C(G), (u_{ij}))$ where $u_{ij}:G\to\mathbb C$ are the evaluation functions $u_{ij}(g)=g_{ij}$ of matrix entries.
If $G = \left( A, u \right)$ and $H = \left( B, v \right)$ are compact matrix quantum groups with $u~\in~M_n(A)$ and $v \in M_n(B)$, we say that $G$ is a 
compact matrix quantum subgroup of $H$, if there is a surjective $^{*}$-isomorphism from $B$ to $A$ mapping generators to generators.
We then write $G \subseteq H$. If we have $G \subseteq H$ and $H \subseteq G$, they are said to be equal as compact matrix quantum groups.

\begin{ex}
        \label{snplus}
        \label{sncomm} 
        An example for a compact matrix quantum group is the quantum symmetric group $S_n^+ = (C(S_n^+), u)$, which was  defined by Wang~\cite{wang} in 1998.
        It is the compact matrix quantum group given by
\[C(S_n^+) := C^*\langle u_{ij}, 1\leq i,j\leq n\;|\; u_{ij} = u_{ij}^* = u_{ij}^2, \sum_{k=1}^{n} u_{ik} = \sum_{k=1}^n u_{ki} = 1\rangle.\]
The quotient of $C\left(S_n^{+}\right)$ by the relation that all $u_{ij}$ commute is exactly $C\left(S_n\right)$.
    We have $S_n \subseteq S_n^{+}$ as compact matrix quantum groups. For $n\leq 3$, we have $S_n=S_n^+$, i.e. $C(S_n^+)$ is commutative. For $n \geq 4$ however, $C(S_n^+)$ is non-commutative as may be seen from the following surjective $*$-homomorphism onto the $C^*$-algebra from Example \ref{ExProj}
        \[\phi: C(S_{4}^+) \rightarrow C^*\langle p, q, 1 \;| \;p = p^* = p^2, q = q^* = q^2\rangle,\]
where we map the generators $u_{ij}\in C(S_{4}^+)$ to $p$, $q$, $1-p$, $1-q$ or $0$ according to the following matrix: 
        \[ \begin{bmatrix} 
                        u_{11}&u_{12}&u_{13}&u_{14}\\
                        u_{21}&u_{22}&u_{23}&u_{24}\\
                        u_{31}&u_{32}&u_{33}&u_{34}\\
                        u_{41}&u_{42}&u_{43}&u_{44}\\
                \end{bmatrix} \mapsto \begin{bmatrix} 
                        p&1-p&0&0\\
                        1-p&p&0&0\\
                        0&0&q&1-q\\
                        0&0&1-q&q\\ 
                \end{bmatrix}
        \]
\end{ex}

\subsection{Automorphism groups of finite graphs}
\label{aut}
We consider finite graphs $\Gamma=(V,E)$ with $V=\{1,\ldots,n\}$ having no multiple edges (i.e. we have $E\subset V\times V$). The adjacency matrix of such a graph is given by $\epsilon\in M_n(\{0,1\})$ with $\epsilon_{ij}=1$ if and only if $(i,j)\in E$. An automorphism of $\Gamma$ is a bijective map $\sigma: V \rightarrow V$ (i.e. an element in $S_n$) such that 
$(\sigma(i), \sigma(j)) \in E$ if and only if $(i, j) \in E$; equivalently we have $\sigma\epsilon\sigma^{-1}=\epsilon$, if we view $\sigma\in S_n$ as the permutation matrix with entries $\sigma_{ij}:=\delta_{\sigma(i)j}$.
The set of all automorphisms of $\Gamma$ forms the automorphism group $\Aut(\Gamma)$. 
It is a subgroup of $S_n$:
\begin{displaymath}
        \Aut(\Gamma) = \{\sigma \in S_n\; |\; \sigma \varepsilon = \varepsilon \sigma \} \subseteq S_n
\end{displaymath}

\subsection{Quantum automorphism groups of finite graphs}
\label{qaut}
Given a finite graph as above, its quantum automorphism group $\QBan(\Gamma)$ is defined \cite{banica} as the 
compact matrix quantum group given by
\[C(\QBan(\Gamma)) = C^*\langle u_{ij}, 1\leq i,j\leq n\;|\; u_{ij} = u_{ij}^*=u_{ij}^2, \sum_{k=1}^n u_{ik} = \sum_{k=1}^n u_{ki} = 1, u \varepsilon = \varepsilon u\rangle.\]
One can show \cite[Lemma 6.7]{speicherweber} that the relations on the generators are equivalent to:
\begin{align}
        &u_{ij} = u_{ij}^*&1 \leq i, j, k \leq n\\
        & u_{ij} u_{ik} = \delta_{jk} u_{ij}, u_{ji} u_{ki} = \delta_{jk} u_{ji} &1 \leq i, j, k \leq n\\
        &\sum_{k=1}^n u_{ik} = \sum_{k=1}^n u_{ki} = 1  &1\leq i \leq n\\
        &u_{ik} u_{jl} = u_{jl} u_{ik} = 0  &(i, j) \in E, (k, l) \notin E\\
        &u_{ik} u_{jl} = u_{jl} u_{ik} = 0  &(i, j) \notin E, (k, l) \in E
\end{align}

Note that we have 
        \[\Aut\left( \Gamma \right) \subseteq \QBan\left( \Gamma \right)\]
        in the sense of Section \ref{cmqg}: The quotient of the (not necessarily commutative) $C^*$-algebra $C(\QBan(\Gamma))$ by the relations $u_{ij}u_{kl}=u_{kl}u_{ij}$ yields $C(\Aut(\Gamma))$.

\begin{defn}
We say that a graph $\Gamma$ has quantum symmetries, if $\Aut\left( \Gamma \right) \subsetneq \QBan\left( \Gamma \right)$, or, equivalently, if $C(\QBan(\Gamma))$ is non-commutative.
\end{defn}

\begin{ex}
        \label{full}
        If $\Gamma$ is the full, undirected graph on $n$ vertices,  we have  $\QBan(\Gamma) = S_n^+$. Hence, for $n\geq 4$, this graph has quantum symmetries.
\end{ex} 
See Section  \ref{Examples} and\cite{simon, simonFolded} for more on quantum automorphism groups of graphs.

\subsection{Criteria for computing quantum symmetry}\label{TwoCriteria}
In this subsection, we collect a number of useful lemmata for computing quantum symmetry.

The first criterion is due to one of the authors of this article, see \cite{simonFolded} and it is applied to the automorphism group of the given graph. Let $\sigma, \tau\in S_n$ be two permutations. We say that they are disjoint, 
if $\sigma(i) \neq i$ implies $\tau(i) = i$ for all $i \in \{1,\ldots,n\}$, and likewise  $\tau(i) \neq i $ implies $\sigma(i) = i$ for all $i$.
\begin{lem}[Disjoint Automorphisms Criterion]
        \label{disjAutom} 
        Let $\Gamma = \left( V, E \right)$ be a finite graph without multiple edges, $V = \left\{ 1, \cdots, n \right\}$. If there are two non-trivial, disjoint automorphisms $\sigma, \tau \in \Aut\left( \Gamma \right)$,         then  $\Gamma$ has quantum symmetries.
\end{lem}
\begin{proof}
We find a surjection  onto the $C^*$-algebra from Example \ref{ExProj}, $\phi: C(\QBan(\Gamma))\to C^*\langle p, q, 1 \;| \;p = p^* = p^2, q = q^* = q^2\rangle$ mapping $u_{ij}\mapsto \delta_{\sigma(i)j} p+\delta_{\tau(i)j} q+ \delta_{ij} (1-p-q)$; see \cite{simonFolded}.
\end{proof}

The second criterion can be found in the PhD thesis of Fulton~\cite{fulton}. It relies on powers of the adjacency matrix.
\begin{lem}[Fulton Criterion]
        \label{epsilon}
        Let $\Gamma$ be a finite graph without multiple edges and let $\varepsilon \in M_n\left( \{0, 1\} \right)$ be its adjacency matrix.
        Denote by $\varepsilon_{ij}^{(l)}$ the $\left( i, j \right)$ entry of the $l$-th power $\varepsilon^l$ of $\epsilon$. If $\varepsilon_{ii}^{\left( l \right)} \neq \varepsilon_{jj}^{\left( l \right)}$
        for some $l \in \N$, then $u_{ij} = 0$ in $C(\QBan(\Gamma))$.
\end{lem}
\begin{proof}
        Assume $\varepsilon_{ii}^{(l)} \neq \varepsilon_{jj}^{(l)}$ for some $l \in \N$. It holds that $u \varepsilon^{l} = \varepsilon^{l} u$ or equivalently
        \[ \sum_{k=1}^n \varepsilon_{ik}^{(l)} u_{kj} = \sum_{k=1}^{n}u_{ik} \varepsilon_{kj}^{(l)}.\]
        Multiplying this equation by $u_{ij}$ yields thanks to Relation (2.2):
        \[ \varepsilon_{ii}^{(l)} u_{ij} = \varepsilon_{jj}^{(l)} u_{ij}\]
        Since $\varepsilon_{ii}^{(l)} \neq \varepsilon_{jj}^{(l)}$, we get $u_{ij} = 0$.
\end{proof}

Observe that the $l$-th power of $\epsilon$ counts the number of paths of length $l$ from one vertex to another, i.e. if $\epsilon_{ij}^{(l)}=k$, then there are exactly $k$ different paths $(v_0,\ldots,v_l)$ with $v_s\in V$, $v_0=i$, $v_l=j$ and $(v_s,v_{s+1})\in E$; we allow $v_s=v_t$ here.

\begin{rem}
Note that the Fulton criterion does not yield any information in the case of vertex-transitive graphs (or more generally, for walk-regular graphs) as we have $\varepsilon_{ii}^{\left( l \right)}=\varepsilon_{jj}^{\left( l \right)}$ for all $i,j,l$ in that case.
\end{rem}

\begin{rem}\label{epsilonForAut}
Note that $u_{ij} = 0$ in $C(\QBan(\Gamma))$ implies $\sigma_{ij}=0$ for all $\sigma\in\Aut(\Gamma)$ in the notation of Section \ref{aut}, since any $\sigma\in\Aut(\Gamma)$ gives rise to a $^*$-homomorphism $\phi:C(\QBan(\Gamma))\to\C$ mapping $u_{ij}\mapsto \sigma_{ij}$.
\end{rem}

Finally, the third criterion is based on the connection between $C^*$-algebras and $\mathbb C$-algebras, see Section \ref{CAlgebras}.
 It may be viewed as a soft algebraisation of quantum automorphism groups of graphs.

\begin{defn}\label{DefAlgQAut}
Let $\Gamma$ be a finite graph as in Section \ref{aut}. We define the following universal unital complex algebra:
\[\AAut(\Gamma):=\mathbb C\langle 
 u_{ij}, 1\leq i,j\leq n\;|\; \textnormal{Relations (2.2)--(2.5)}\rangle\]
\end{defn}

\begin{lem}[Algebraic Criterion]\label{LemLink}
Let $\Gamma$ be a finite graph on $n$ vertices without multiple edges. If $\AAut(\Gamma)$ is commutative, then so is $C(\QBan(\Gamma))$, i.e. $\Gamma$ has no quantum symmetries.
\end{lem}
\begin{proof}
We have an algebra homomorphism from $\AAut(\Gamma)$ to $C(\QBan(\Gamma))$ which has dense image. Hence, the statement follows.
\end{proof}

\begin{rem}\label{RemAlgCriterion}
The Algebraic Criterion basically states that if we may derive the relations $u_{ij}u_{kl}=u_{kl}u_{ij}$ from Relations (2.2)--(2.5) by purely algebraic means (i.e. in $\AAut(\Gamma)$), then we may do so also in $C(\QBan(\Gamma))$. The converse does not hold in general: The commutativity relations might follow from Relations (2.1)--(2.5) by using some purely $C^*$-algebraic techniques. For instance, we have $x^*x=0$ if and only if $x=0$ if and only if $x^*=0$ in any $C^*$-algebra as a result from the norm conditions. As a concrete example (see \cite{SimonNeu}), if we have  $u_{ij}u_{kl}u_{ij}=u_{ij}u_{kl}$ in $C(\QBan(\Gamma))$, then
\[u_{ij}u_{kl}= u_{ij}u_{kl}u_{ij}=(u_{ij}u_{kl}u_{ij})^*=(u_{ij}u_{kl})^*=u_{kl}u_{ij}.\]
\end{rem}

\section{Examples}\label{Examples}
For the convenience of the reader, we list a couple of examples. In the remainder of the article, we focus on finite, undirected (i.e. we have $(i,j)\in E$ implies $(j,i)\in E$), connected graphs, having no multiple edges (i.e. we have $E\subset V\times V$) and no loops (i.e. we have $(i,i)\notin E$).

\subsection{Circle graphs}\label{CircleGraphs}
        Let $\Gamma$ be the circle on four vertices:
        
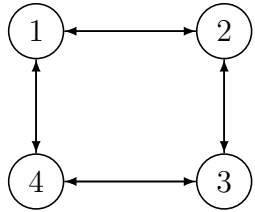
\begin{center}
        \begin{tikzpicture}[-latex ,auto ,node distance = 2 cm and 2.5 cm, on grid,semithick, state/.style={circle, top color = white, bottom color = white, draw,black, text=black, minimum width = 0.15cm}]
        \node[state] (A) {$1$};
        \node[state] (B) [right=of A] {$2$};
        \node[state] (C) [below=of B] {$3$};
        \node[state] (D) [left=of C] {$4$}; 
        \path (A) edge (B);
        \path (B) edge (A);
        \path (C) edge (B); 
        \path (B) edge (C); 
        \path (D) edge (C); 
        \path (C) edge (D); 
        \path (A) edge (D); 
        \path (D) edge (A);

        \end{tikzpicture}
        \end{center} 
        By the Disjoint Automorphism Criterion (Lemma \ref{disjAutom}) it has quantum symmetries, since $\sigma := (1, 3) \in \Aut(\Gamma)$ and $\tau := (2, 4) \in \Aut(\Gamma)$
      are disjoint.
        In fact \cite{bichon}, $\QBan(\Gamma) = H_2^+$, where $H_2^+$ is the hyperoctahedral quantum group.
        In contrast to the circle on four vertices, the circle on $n\neq 4$ vertices does \emph{not} have quantum symmetries \cite{banica}.

\subsection{Santa's house}

        \label{nikolaus}
        Let $\Gamma$ be the following graph:
         
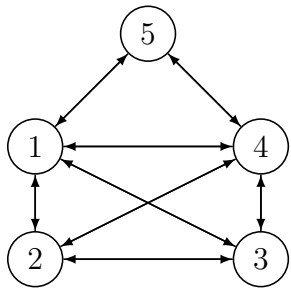
\begin{center}
\begin{tikzpicture}[-latex ,auto ,node distance = 1.5 cm and 1.5 cm, on grid,semithick, state/.style={circle, top color = white, bottom color = white, draw,black, text=black, minimum width = 0.15cm}]
\node[state] (A) {$5$};
\node[state] (B) [below left=of A] {$1$};
\node[state] (C) [below right=of A] {$4$};
\node[state] (D) [below=of B] {$2$};
\node[state] (E) [below=of C] {$3$};
\path (A) edge (B);
\path (A) edge (C);
\path (B) edge (D);
\path (B) edge (C);
\path (C) edge (B);
\path (E) edge (D);
\path (D) edge (E);

\path (B) edge (E);
\path (C) edge (D);
\path (C) edge (E);
\path (B) edge (A);
\path (D) edge (B);
\path (D) edge (C);
\path (C) edge (A);
\path (E) edge (B);
\path (E) edge (C);
\end{tikzpicture}
\end{center} 
        We have that $(2, 3) \in \Aut(\Gamma)$ and $(1, 4) \in \Aut(\Gamma)$. The Disjoint Automorphism Criterion thus yields that $\Gamma$ does have quantum symmetries.
        Now we want to compute the quantum automorphism group of $\Gamma$. We have that 
        \[\varepsilon = \begin{bmatrix}
                        0&1&1&1&1\\
                        1&0&1&1&0\\
                        1&1&0&1&0\\
                        1&1&1&0&1\\
                        1&0&0&1&0\\
                \end{bmatrix} \text{ and } 
                \varepsilon^2 = \begin{bmatrix}
                        4&2&2&3&1\\
                        2&3&2&2&2\\
                        2&2&3&2&2\\
                        3&2&2&4&1\\
                        1&2&2&1&2\\
                \end{bmatrix}
        \]
        Thus, by the Fulton Criterion (Lemma \ref{epsilon}) and Relation (2.3) we have that the generating matrix $u$ of $C(\QBan(\Gamma))$ looks as follows:
        \[u = \begin{bmatrix}
                        u_{11}&0&0&1-u_{11}&0\\
                        0&u_{22}&1-u_{22}&0&0\\
                        0&1-u_{22}&u_{22}&0&0\\
                        1-u_{11}&0&0&u_{11}&0\\
                        0&0&0&0&1\\
                \end{bmatrix}
        \]
Since  $\Gamma$  has quantum symmetries, we conclude that $\QBan(\Gamma) = \widehat{\Z_2 * \Z_2}$ (which means $C(\QBan(\Gamma)) = C^*(\Z_2 * \Z_2)$ on the level of $C^*$-algebras), compare also \cite[Thm 3.8]{simon}.

\subsection{Santa's house with broken roof}
        If we now look at the above graph after taking away the (undirected) edge $(4,5)$, we obtain:
        
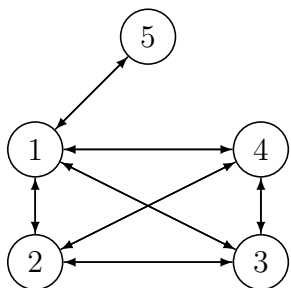
\begin{center}
        \begin{tikzpicture}[-latex ,auto ,node distance = 1.5 cm and 1.5 cm, on grid,semithick, state/.style={circle, top color = white, bottom color = white, draw,black, text=black, minimum width = 0.15cm}]
        \node[state] (A) {$5$};
        \node[state] (B) [below left=of A] {$1$};
        \node[state] (C) [below right=of A] {$4$};
        \node[state] (D) [below=of B] {$2$};
        \node[state] (E) [below=of C] {$3$};
        \path (A) edge (B);
        \path (B) edge (D);
        \path (B) edge (C);
        \path (C) edge (B);
        \path (E) edge (D);
        \path (D) edge (E);
        \path (B) edge (E);
        \path (C) edge (D);
        \path (C) edge (E);
        \path (B) edge (A);
        \path (D) edge (B);
        \path (D) edge (C);
        \path (E) edge (B);
        \path (E) edge (C);
        \end{tikzpicture}
        \end{center}
        and thus 
        \[\varepsilon = \begin{bmatrix}
                        0&1&1&1&1\\
                        1&0&1&1&0\\
                        1&1&0&1&0\\
                        1&1&1&0&0\\
                        1&0&0&0&0\\
                \end{bmatrix} \text{ and } 
                \varepsilon^2 = \begin{bmatrix}
                        4&2&2&2&0\\
                        2&3&2&2&1\\
                        2&2&3&2&1\\
                        2&2&2&3&1\\
                        0&1&1&1&1\\
                \end{bmatrix}
        \]
        By the Fulton Criterion we obtain
        \[u = \begin{bmatrix} 
                        1&0&0&0&0\\
                        0&u_{22}&u_{23}&u_{24}&0\\
                        0&u_{32}&u_{33}&u_{34}&0\\
                        0&u_{42}&u_{43}&u_{44}&0\\
                        0&0&0&0&1\\
                \end{bmatrix}
        \]
  and hence $\QBan(\Gamma) = \Aut(\Gamma) = S_{3}$ (recall that $S_3^+=S_3$), which means that the graph does not have quantum symmetries.

\subsection{A graph with trivial quantum automorphism group}
For the graph
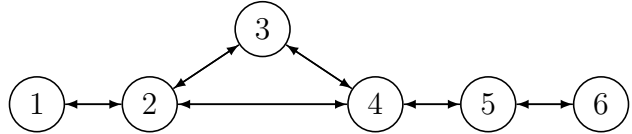
\begin{center}
\begin{tikzpicture}[-latex ,auto ,node distance = 1 cm and 1.5 cm, on grid,semithick, state/.style={circle, top color = white, bottom color = white, draw,black, text=black, minimum width = 0.15cm}]
\node[state] (A) {$1$};
\node[state] (B) [right=of A] {$2$};
\node[state] (C) [above right=of B] {$3$};
\node[state] (D) [below right=of C] {$4$};
\node[state] (E) [right=of D] {$5$};
\node[state] (F) [right=of E] {$6$}; 
\path (A) edge (B);
\path (B) edge (A);
\path (C) edge (B); 
\path (B) edge (C);
\path (B) edge (D); 
\path (D) edge (B);
\path (D) edge (C); 
\path (C) edge (D); 
\path (E) edge (D);
\path (F) edge (E);

\path (D) edge (E);
\path (E) edge (F);
\end{tikzpicture}
\end{center} 
we have
\begin{displaymath}
        \varepsilon = \begin{bmatrix}
                        0&1&0&0&0&0\\
                        1&0&1&1&0&0\\
                        0&1&0&1&0&0\\
                        0&1&1&0&1&0\\
                        0&0&0&1&0&1\\
                        0&0&0&0&1&0\\
                \end{bmatrix}
        \text{ and }
        \varepsilon^4 = \begin{bmatrix}
                        3&2&4&5&1&1\\
                        2&12&7&7&6&1\\
                        4&7&8&7&5&1\\
                        5&7&7&13&2&4\\
                        1&6&5&2&6&0\\
                        1&1&1&4&0&2\\
                \end{bmatrix} 
\end{displaymath}
By the Fulton Criterion we deduce $u_{ij} = \delta_{ij}$ and thus $C(\QBan(\Gamma))$ is one-dimensional; hence $\QBan(\Gamma) = \Aut(\Gamma) = \{e\}$.

\section{A computational approach to quantum symmetries of graphs}

We now discuss our approach for checking the existence of quantum symmetries.

\subsection{Preprocessing in \textsc{GAP} and \textsc{PYTHON}}
We used  \textsc{Python} scripts for producing the adjacency matrices of graphs on $n\leq 7$ vertices as input files and we restricted to equivalence classes of graphs, with the help of \textsc{Gap} \cite{gap}. We  used \textsc{Gap} again  to calculate the automorphism group of every graph we considered.

\subsection{Implementation of the Algebraic Criterion in \textsc{SINGULAR}} \quad


\begin{algorithm}[H]
        \caption{\textsc{QSym}}
        \LinesNumbered
        \SetKwData{Left}{left}\SetKwData{This}{this}\SetKwData{Up}{up}
        \SetKwFunction{Union}{Union}\SetKwFunction{FindCompress}{FindCompress}
        \SetKwInOut{Input}{input}\SetKwInOut{Output}{output}
        \Input{$\epsilon$, the adjacency matrix of a graph $\Gamma$ on $n$ vertices}
        \Output{the value 1, if $\AAut(\Gamma)$ is commutative\\ the value 0  otherwise}
        let $R$ be a $\mathbb C$-algebra with generators $u_{ij}, i, j = 1, \cdots, n$\;
        let $I$ be an empty ideal in $R$\;
        let $J$ be an empty ideal in $R$\;
        \For{$i, j, k = 1 .. n$}{
            add the relation $u_{ik} u_{jk} = \delta_{ij} u_{ik}$ to the ideal $I$\;
            add the relation $u_{ki} u_{kj} = \delta_{ij} u_{ki}$ to the ideal $I$\;
         }
        \For{$i = 1 .. n$}{
                add the relation $\sum_{j=1}^{n} u_{ij} = 1$ to the ideal $I$\;
                add the relation $\sum_{j=1}^{n} u_{ji} = 1$ to the ideal $I$\;
        }
        \For{$i, j, k, l = 1 .. n$}{
                \eIf{$\epsilon[i, j] = 1$}{
                        \If{$\epsilon[k, l] \neq 1$}{
                            add the relation $u_{ik} u_{jl} = 0$ to the ideal $I$\;}
                        }{
                            \If{$\epsilon[k, l] = 1$}{
                                add the relation $u_{ik} u_{jl} = 0$ to the ideal $I$\;
                        }
                    }
        } 
               \For{$k = 1 .. n^2$}{
            \For{$i, j = 1 .. n$}{ 
                    \If{$\epsilon^{(k)}[i, i] \neq \epsilon^{(k)}[j, j]$}{add the relation $u_{ij} = 0$ to the ideal $I$\;} 
            } 
        }
        \For{$i, j, k, l = 1 .. n$}{
                add the relation $u_{ij} u_{kl} = u_{kl} u_{ij}$ to the ideal $J$ \; 
        }
        compute a two-sided  Gr\"obner basis of $I$\;
        \For{$t$ in $J$}{
            \If{$t$ is not in $I$}{return 0 and exit\;}
        }
        return 1\; 
\end{algorithm}

 We  implemented the algorithm \textsc{QSym} in \textsc{Singular} \cite{singular} using the subsystem \textsc{Letterplace} \cite{letterplace}. The algorithm is  a straightforward implementation of the algebra $\AAut(\Gamma)$ from Definition \ref{DefAlgQAut}, i.e. of the complex algebra given by  Relations (2.2)--(2.5) from Section \ref{qaut}; the Fulton Criterion (Lemma \ref{epsilon}) is also inserted.
 
Indeed, we first produce the free $\mathbb C$-algebra $R$ with generators $u_{ij}$, for $1\leq i,j\leq n$. We then add all relations of $\AAut(\Gamma)$ to the ideal $I$, i.e. Relations (2.2) in lines 4--6, Relations (2.3) in lines 7--9, Relations (2.4) in lines 10--13 and Relations (2.5) in lines 14--16. 
Note that for Relations (2.4) and (2.5) we do not need to add the relations $u_{jl}u_{ik}=0$ since every possible combination of the indices $i,j,k$ and $l$ will pass through the loop of lines 10--16.
In lines 17--20, we add relations $u_{ij}=0$ whenever the Fulton Criterion (Lemma \ref{epsilon}) is satisfied for powers of the adjacency matrix up to the power $n^2$ (which is an ad hoc choice for a  bound). In lines 21--22 we construct the ideal $J$ of all commutativity relations. We then compute the Gr\"obner basis with respect to the degree reverse lexicographical monomial order of the ideal $I$ in $R$ and check whether there is any element from $J$ which is not in $I$, see lines 23--25. If so, the algorithm terminates and gives the value 0 as an output; otherwise, the output is 1.

\subsection{Limits of the algorithm \textsc{QSYM}}\label{limits}

Regarding the limits of the algorithm \textsc{QSym}, let us note that since it is based on Gr\"obner basis computations, there are natural limits for the number of generators, also depending on the complexity of the adjacency matrix $\epsilon$. In other words: The situation may occur, where we simply cannot compute the Gr\"obner basis of the ideal, i.e. we are stuck before coming to lines 23--25 of our algorithm. This did not happen for the graphs we considered, but computations for a higher number of vertices (such as $n=10$) quickly reach this point.

One  might improve the algorithm by playing around with different monomial orders for computing the Gr\"obner basis, by first applying the Fulton Criterion before running the algorithm (and thus by using less variables in the first place), by parallelizing the computation of the Gr\"obner basis, or by improving the Algebraic Criterion, i.e by adding further relations coming from the $C^*$-algebraic side, see Remark \ref{RemAlgCriterion}. The main factor is, of course, the speed and memory consumption of the Gr\"obner base computation for noncommuting elements, as implemented by \textsc{Letterplace} in \textsc{Singular}.

\subsection{Implementation of the Disjoint Automorphism Criterion in \textsc{GAP}}
We  used \textsc{Gap}   to  check the Disjoint Automorphism Criterion (Lemma \ref{disjAutom}),
i.\,e. whether the automorphism group contains two disjoint permutations.

\subsection{Conclusion: a combination of two tools}

Let us summarize the statements from the Disjoint Automorphism Criterion (Lemma \ref{disjAutom}) and the Algebraic Criterion (Lemma \ref{LemLink}).

\begin{lem}[{Lemma \ref{disjAutom} and Lemma \ref{LemLink}}]
Let $\Gamma$ be a finite graph having no multiple edges.
\begin{itemize}
\item[(a)] If the Disjoint Automorphism Criterion applies, then $\Gamma$ has quantum symmetries.
\item[(b)] If \textsc{QSym} has output 1, then $\Gamma$ has no quantum symmetries.
\end{itemize}
\end{lem}

Note that the situation might occur, when there are no disjoint automorphisms and \textsc{QSym} has output 0. In that case, no conclusion is possible for the existence of quantum symmetries.

\section{Application of the algorithm to graphs on up to seven vertices}

We now sketch how we applied our algorithm in a concrete setup.

\subsection{The data}\label{Data}

As before, we consider undirected graphs $\Gamma$ on $n$ vertices having no multiple edges and no loops. For $n\leq 3$, we have $S_n=S_n^+$ and hence $\Aut(\Gamma)=\QBan(\Gamma)$, i.e. we never have quantum symmetries. The case of graphs on $n=4$ vertices has been treated in \cite{simon}. As for $n=5$ and $n=6$, we checked for every connected graph both the Disjoint Automorphism Criterion as well as the Algebraic Criterion, the latter one via the algorithm \textsc{QSym}. Fortunately, the criteria matched perfectly well, i.e. the Disjoint Automorphism Criterion applied if and only if the output of \textsc{QSym} was 0.
Hence, combining these two tools, we were able to settle the question of the existence of quantum symmetries for $n=5$ and $n=6$.
For $n=7$, we also run our algorithm on graphs whose automorphism groups are of order one or two, \textsc{QSym} being constantly 1.

In the following table (from the introduction) we list all possible orders of the automorphism groups, the number of graphs whose automorphism group has this order, and the number of graphs amongst them having quantum symmetries.

\begin{table}[H]
        \centering\label{table1}
        \begin{tabular}{|c||c c|c c|c c|c c|}
                \hline
                Order&\multicolumn{2}{c|}{4 vertices}&\multicolumn{2}{c|}{5 vertices} & \multicolumn{2}{c|}{6 vertices} & \multicolumn{2}{c|}{7 vertices}\\
                of $\Aut(\Gamma)$&total& qsym&total& qsym&total& qsym&total& qsym\\
                \hline
                720&&& & &1&1&&\\
                120&&&1&1&1&1&&\\
                72&&&&&1&1&&\\
                48&&&&&4&4&&\\
                36&&&&&1&1&&\\
                24&1&1&1&1&1&1&&\\
                16&&&&&3&3&&\\
                12&&&3&3&10&8&&\\
                10&&&1&0&1&0&&\\
                8&1&1&2&2&9&9&&\\
                6&1&0&1&0&7&0&&\\
                4&1&1&3&3&28&26&&\\
                2&2&0&9&0&37&0&317&0\\
                1&0&0&0&0&8&0&144&0\\
                total&6&3&21&10&112&55&853&?\\
                \hline
        \end{tabular}
        \caption{Number of connected, undirected graphs with quantum symmetries}
\end{table}

\subsection{Some observations}\label{SectObservations}

Studying the table, we notice a couple of things:
\begin{enumerate}
\item The ratio between graphs having quantum symmetries and those having no quantum symmetries is about $50:50$. As it is known that almost all graphs have no symmetries and no quantum symmetries \cite{LMR}, this effect is just a distortion for small $n$. However, it could be interesting to observe until which $n$ this phenomenon occurs.
\item No graph whose automorphism group has order one or two has quantum symmetries. See also Theorem \ref{mainThm}.
 Moreover, we observe that also the order 6 might be an obstruction for quantum symmetries.
\item Conversely, the orders 4, 8 and 12 seem to be quite friendly towards the existence of quantum symmetries.
\end{enumerate}

\subsection{From connected graphs to all graphs}

Before we come to some conclusion derived from the above data on connected graphs, let us collect some facts on disconnected graphs. We first recall the disjoint union of graphs.

Let $\Gamma=(V,E)$ be a graph. By $\Gamma^{\sqcup m}$ we denote the disjoint union of $m$ copies of $\Gamma$, i.e. its vertex set can be written as $V^{\sqcup m}=V\times\{1,\ldots,m\}$ and vertices $(i,x)\in V^{\sqcup m}$ and $(j,y)\in V^{\sqcup m}$ are connected by an edge if and only if $x=y$ and $(i,j)\in E$. Let us remark that if $\Gamma$ consists in a single vertex, then $\Aut(\Gamma^{\sqcup m})=S_m$, whereas $\QBan(\Gamma^{\sqcup m})=S_m^+$. If $\Gamma$ is the graph on two connected vertices, then
\[\Aut(\Gamma^{\sqcup m})=H_m:=\Z_2\wr S_m=(\Z_2\times \ldots \times \Z_2)\rtimes S_m,\]
where $H_m$ denotes the hyperoctahedral group, 
the group obtained from the canonical action of the symmetric group $S_m$ on $m$ copies of the cyclic group $\Z_2$ on two generators. By $H_m^+$ we denote its quantum analog  (see also Example \ref{CircleGraphs})
and we have $\QBan(\Gamma^{\sqcup m})=H_m^+$ in that case \cite{bichon}.

\begin{lem}\label{LemComponents}
Let $\Gamma_0=(V_0,E_0)$ be a connected graph and let $\Gamma_1=\Gamma_0^{\sqcup m}$. Let $\Gamma_2$ be a graph containing no copy of $\Gamma_0$ as a connected component. Let $\Gamma=\Gamma_1\sqcup\Gamma_2$ be the disjoint union of these two graphs. If $|V_0|\in\{1,2\}$, then
\[\Aut(\Gamma)=\Aut(\Gamma_1)\times \Aut(\Gamma_2)\qquad\textit{and}\qquad\QBan(\Gamma)=\QBan(\Gamma_1)*\QBan(\Gamma_2).\] 
\end{lem}
\begin{proof}
Consider $\Gamma=(V,E)$ and denote by $V_k\subset V$ the vertices coming from $\Gamma_k$, for $k=1,2$. Let $i\in V_1$ and $j\in V_2$. We are going to apply the  criterion  from Lemma \ref{epsilon} and Remark \ref{epsilonForAut} in order to show that $u_{ij}=0$ resp. $\sigma_{ij}=0$ for $\sigma\in\Aut(\Gamma)$. This then proves the decomposition results, since we then have $u=u_1\oplus u_2$ for the matrix $u$ of $\Aut(\Gamma)$ and $\QBan(\Gamma)$ respectively. This implies that we find homomorphisms between the corresponding $C^*$-algebras, which are inverse to each other.
 Hence, all  we need to do is to find some $l$ such that $\epsilon_{ii}^{(l)}\neq \epsilon_{jj}^{(l)}$.

In the case $|V_0|=1$, the vertex $i\in V_1$ is not connected to any other vertex. Thus  there are no paths of length $l$ from $i$ to $i$ and we infer $\epsilon_{ii}^{(l)}=0$ for all $l\geq 1$. As for $j\in V_2$, this vertex is connected to at least another vertex $k\in V_2$, since $\Gamma_2$ does not contain any copy of $\Gamma_0$, i.e. it does not contain isolated points, if $|V_0|=1$. Hence, $(j,k,j)$ is a path of length two going from $j$ to $j$, which shows $\epsilon_{jj}^{(2)}\geq 1$ and hence $\epsilon_{ii}^{(2)}\neq \epsilon_{jj}^{(2)}$.

In the case $|V_0|=2$, the vertex $i\in V_1$ is connected to exactly one other vertex. Thus, $\epsilon_{ii}^{(4)}=1$. As for $j\in V_2$ however, this vertex is either connected to no vertex at all -- in which case $\epsilon_{jj}^{(4)}=0$ -- or it is connected to another vertex $k\in V_2$. Now, either $j$ or $k$ is connected to a third vertex $t\in V_2$ since $\Gamma_2$ does not contain any copy of $\Gamma_0$. From this, we may deduce $\epsilon_{jj}^{(4)}\geq 2$ and summarizing, $\epsilon_{ii}^{(4)}\neq \epsilon_{jj}^{(4)}$.
\end{proof}

It should be possible to push further the above considerations, for larger cardinalities $|V_0|$; however, a general decomposition statement for the quantum automorphism group must fail as there are graphs $\Gamma_1$ and $\Gamma_2$ which are quantum isomorphic but not isomorphic, compare the work in \cite{LMR}. In this case, we have
\[\Aut(\Gamma)=\Aut(\Gamma_1)\times \Aut(\Gamma_2)\qquad\textnormal{but}\qquad\QBan(\Gamma)\neq\QBan(\Gamma_1)* \QBan(\Gamma_2).\]

\subsection{No quantum symmetries for graphs with small automorphism groups}

We may derive the following consequence from the above data.

\begin{thm}
        \label{mainThm}
        Let $\Gamma$ be an undirected graph on $n\leq 7$ vertices having no multiple edges and no loops. Then:
        \begin{displaymath}
                \Aut(\Gamma) = \mathbb{Z}_2 \qquad \Rightarrow \qquad \QBan\left( \Gamma \right) = \mathbb{Z}_2
        \end{displaymath}
        \begin{displaymath}
                \Aut(\Gamma) = \left\{  e\right\}\qquad  \Rightarrow \qquad \QBan\left( \Gamma \right) = \left\{ e\right\}
        \end{displaymath}
\end{thm}
\begin{proof}
We can rephrase the statement: If the order of $\Aut(\Gamma)$ is one or two, then the graph has no quantum symmetries.
For $n\leq 3$, we have $S_n=S_n^+$ and hence no graphs has quantum symmetries. For $n=4$, the rephrased statement follows from \cite[Thm. 3.8]{simon}. For $n\in\{5,6,7\}$, assume first that $\Gamma$ is connected. Then, the rephrased statement follows from our computer based verification, see Section \ref{Data}. 

Now, if $\Gamma$ is disconnected, firstly assume that we have exactly one isolated vertex. By Lemma \ref{LemComponents}, we have $\Aut(\Gamma)=\Aut(\Gamma')$ and $\QBan(\Gamma)=\QBan(\Gamma')$ where $\Gamma'$ is obtained from removing this isolated vertex. Hence, we are in the case of a graph on $n-1$ vertices and we may proceed inductively. Secondly, if $\Gamma$ does not have exactly one isolated vertex, it decomposes into two graphs $\Gamma_1$ and $\Gamma_2$ with no edges between them and both graphs having at most five vertices. Thus neither $\Aut(\Gamma_1)$ nor $\Aut(\Gamma_2)$ has order one (see also the appendix for a list of concrete graphs) and we find two disjoint automorphisms of $\Gamma$ in the sense of Section \ref{TwoCriteria}. This proves that $\Aut(\Gamma)$ has order at least four, so we don't need to take care of this case.
\end{proof}

In the case $\Aut(\Gamma) = \left\{  e\right\}$ the result from the above theorem may also be obtained by other means: Using the Weisfeiler-Lehman algorithm for computing the coherent algebras of asymmetric graphs, one can show that a trivial automorphism group implies the quantum automorphism group being trivial, for all graphs up to ten vertices, building on a recent result by Lupini, Mancinska and Roberson \cite{LMR}; we thank Luca Junk for performing these computations on the computer. However, the criterion on coherent algebras fails in the case  $\Aut(\Gamma) =\mathbb Z_2$.

We are wondering (see also Section \ref{SectObservations}) whether certain orders of automorphism groups are an obstruction for the existence of quantum symmetries: It might be the case that  $\Aut(\Gamma)\in\left\{\{e\},\Z_2,S_3\right\}$  implies $\QBan(\Gamma)=\Aut(\Gamma)$. We have to leave the question open, whether this holds in general.

\bibliography{ComputingQSym}
\bibliographystyle{alpha}

\section{Appendix: Lists of graphs on a small number of  vertices}\label{List}

We finish this article by listing all connected, undirected graphs on a small number of vertices having no multiple edges and no loops. In each case, we depict the graph, we state its automorphism group and its order, we give the information whether or not it is regular, we state the output value of the \textsc{QSym} algorithm in the form ``yes'' (output value 0) or  ``no'' (output value 1). Note that the computation of the Gr\"obner bases was always successful; for each graph the computations took less than an hour on our desktop computer. We also list the information whether or not the Disjoint Automorphism Criterion (Lemma \ref{disjAutom}) is satisfied; recall that ``disj auts: yes'' implies that the graph has quantum symmetries.

\subsection{List of graphs on four vertices}

\quad

\setlength\LTleft{-2cm}


\subsection{List of graphs on five vertices} \quad

\setlength\LTleft{-2cm}
%

\subsection{List of graphs on six vertices} \quad

\setlength\LTleft{-2cm}
%

\end{document}